\documentclass[a4paper,10pt,fleqn]{amsart}

\usepackage{amsmath}
\usepackage{amsfonts}
\usepackage{amsthm}
\usepackage{amsaddr}
\usepackage[utf8]{inputenc}
\usepackage{bbm}

\newtheorem{theorem}{Theorem}[section]
\newtheorem{lemma}[theorem]{Lemma}
\newtheorem{prop}[theorem]{Proposition}

\theoremstyle{remark}
\newtheorem{rem}[theorem]{Remark}

\theoremstyle{definition}
\newtheorem{definition}[theorem]{Definition}
\newtheorem{example}[theorem]{Example}

\DeclareMathOperator{\Dom}{Dom}
\DeclareMathOperator{\Bor}{Bor}
\DeclareMathOperator{\id}{id}
\DeclareMathOperator{\supp}{supp}
\DeclareMathOperator{\re}{Re}
\DeclareMathOperator*{\wlim}{w-lim}
\newcommand{\dt} {\partial_t}
\newcommand{\X} {{\mathbb{R}^d}}
\newcommand{\R} {\mathbb{R}}
\newcommand{\C} {\mathbb{C}}
\newcommand{\N} {\mathbb{N}}
\newcommand{\T} {[0,\infty)}
\newcommand{\Tb} {[0,T]}
\newcommand{\DT} {C_c^\infty(\T\times\X)}
\newcommand{\CTL} {C_c(\T,\F{\sigma_0})}
\newcommand{\intT}[1] {\int_0^\infty #1\,dt}
\newcommand{\iprod}[2] {#1\cdot#2}
\newcommand{\FT}[1] {\mathcal{F}\big\{ #1 \big\}}
\newcommand{\B} {(\dt + \mathcal{L}^*)}
\renewcommand{\phi} {\varphi}
\newcommand{\LL} {\mathcal{L}}
\renewcommand{\P} {\mathcal{P}}
\newcommand{\Q} {\mathcal{Q}}
\newcommand{\A} {\mathcal{A}}
\newcommand{\eL}[1] {e^{#1\LL^*}}
\newcommand{\F}[1] {H(#1)}
\newcommand{\phiT} {\nu_{T^\beta}\phi}
\renewcommand{\=} {\overset{d}{=}}
\newcommand{\rfb}[1] {{\upshape(\ref{#1})}}

\begin{document}
\title[]{Scaling Limits of Solutions of
Linear Evolution Equations
with Random Initial Conditions}
\author{Miłosz Krupski}
\address{Uniwersytet Wrocławski, Instytut Matematyczny\\ pl. Grunwaldzki 2/4, 50-384 Wrocław}
\email{milosz.krupski@math.uni.wroc.pl}
  \begin{abstract}
    We consider a linear equation $\dt u = \LL u $, where $\LL$ is a generator of a semigroup of linear operators on a certain
    Hilbert space related to an initial condition $u(0)$ being a generalised stationary random field on $\X$.
    We show the existence and uniqueness of generalised solutions to such initial value problems.
    Then we investigate their scaling limits.
  \end{abstract}
  \maketitle
  \begin{section}{Introduction}
    Partial differential equations with random initial data appear in a wide variety of topics, e.g in the study
    of the Large Scale Structure of the Universe or growing interfaces in deposition of chemical substances
    \cite{MR1305783}, \cite{MannJr.2001}, \cite{MR1732301}.

    In particular, such equations have been studied
    in a number of papers which were primarily concerned with calculating scaling limits of solutions
    \cite{MR1632518}, \cite{MR1939652}, \cite{MR1828776}, \cite{MR1859007}.
    However, they all lack a formal definition of solutions and instead depend only on classical
    formulas for non-random equations, extended to accommodate random fields. Furthermore, the absence of a proper definition
    prevented the authors from exploring conditions for the existence and uniqueness of solutions.

    Moreover, in the context of scaling limits of solutions, it is a welcome result to show that the obtained limit (provided it exists) is also a solution
    to some problem related to the original one, perhaps only in a less regular sense.
    Looking at previous results, e.g.~\cite{MR1632518}, \cite{MR1939652}, one may indeed see this being the matter; however in most cases,
    the concept of a generalised random field is necessary to properly define initial conditions of those derived problems.

    In this paper we put forward a comprehensive approach to study general linear evolution equations supplemented with random initial conditions,
    using tools provided by the theory of generalised random fields. We consider the following abstract initial value problem
    \begin{equation*}
      \left\{
      \begin{aligned}
        &\dt u = \LL u\quad&\text{on $\T\times\X$}, \\
        &u(0) = \eta_0\quad&\text{on $\X$},
      \end{aligned}
      \right.
    \end{equation*}
    where $\eta_0$ is a broad-sense stationary generalised random field and $\LL$ is a linear operator generating a $C_0$-semigroup of linear operators
    on a certain Hilbert space corresponding to $\eta_0$.

    In Section~2 we prepare tools necessary to develop our theory. Section~3 contains main results, including
    the definition of solutions and a proof that, under reasonable assumptions, a kind of a ``semigroup solution'' is in fact unique. Section~4
    is concerned with calculating scaling limits of solutions in an abstract setting. Finally, Section~5 contains examples
    illustrating the results.

    The scope of this work is limited to linear problems; however, the ideas it contains seem to be applicable to
    non-linear problems as well. Such considerations will be the subject of our forthcoming papers.
  \end{section}
  \begin{section}{Preliminaries}
    \begin{subsection}{Basic notation}
      We denote the Borel sigma-algebra on $\X$ by $\Bor(\X)$ and the Lebesgue measure by either $\lambda$ or $dx$.

      We use the Fourier transform defined as
      \begin{equation*}
        \FT{f}(\xi) = \int_\X e^{i\iprod{\xi}{x}}f(x)\,dx.
      \end{equation*}

      Given a measure space $(X,\Theta,\mu)$, by $L^2(X,\Theta,\mu)$ we denote the space of all $\Theta$-measurable real functions such that the integral
      \begin{equation*}
        \int_X |f|^2 d\mu = \int_X |f(x)|^2 \mu(dx)
      \end{equation*}
      is finite. Usually we will shorten the notation to $L^2(\mu) = L^2(X,\Theta,\mu)$.

      Let us fix a probability space $(\Omega,\Sigma,P)$ and denote $L^2(\Omega) = L^2(\Omega,\Sigma,P)$.
      This is the set of all random variables with bounded variations.
      It is a Hilbert space with the standard inner product $EXY$ defined for all $X,Y\in L^2(\Omega)$.

      We also use the notation $L^2(X,\Theta,\mu;\C)$, $L^2(\Omega;\C)$ to describe analogous spaces of complex-valued functions.

      We write
      \begin{equation*}
        (X_1,\ldots,X_k) \= (Y_1,\ldots,Y_k),\quad\text{where $X_i,Y_i\in L^2(\Omega)$ and $k\in\N$}
      \end{equation*}
      if both random vectors have the same probability distributions and we say that
      \begin{equation*}
        \wlim_{n\to\infty}(X_{1,n},\ldots,X_{k,n}) \= (Y_1,\ldots,Y_k)
      \end{equation*}
      if the corresponding probability distributions converge weakly as measures on the product space.

      For a linear operator $\A$, by $\Dom(\A)$ we denote its domain.
      Let us also recall a definition of an admissible subspace following \cite[Ch.~4, Def.~5.3]{MR710486}.
      \begin{definition}\label{admissibility}
        Let $Y$ be a linear subspace of a Banach space $(X,\|\cdot\|_X)$ such that there exists a norm $\|\cdot\|_Y$ in which $Y$ is a Banach space itself, and
        $\|y\|_X \leq C\|y\|_Y$ for all $y\in Y$.
        Suppose $\A$ is the generator of a $C_0$-semigroup $T(t)$ on $X$.
        A subspace $Y$ of $X$ is called $\A$-admissible if it is an invariant subspace of $T(t)$, $t \geq 0$,
        and the restriction of $T(t)$ to $Y$ is a $C_0$-semigroup on $Y$ (i.e.~it is strongly continuous in the norm $\|\cdot\|_Y$).
      \end{definition}
    \end{subsection}
    \begin{subsection}{Generalised random fields}
      Now we recall the basics of the theory of generalised random fields. For a more detailed description we refer
      the reader to \cite{MR0173945}, \cite{MR1629699} and \cite{MR0097842}.

      First we introduce the space of test functions $D=C_c^\infty(\X)$ consisting of smooth functions with compact supports.
      \begin{definition}
        A sequence $\phi_n$ is convergent to $\phi$ in $D$ if
        \begin{enumerate}
          \item there exists a compact $K\subset\X$ such that $\supp\phi\cup\bigcup_{n=0}^\infty\supp \phi_n \subset K$
          \item derivatives $\partial^\alpha\phi_n$ converge uniformly to $\partial^\alpha\phi$ for every multi-index $\alpha$.
        \end{enumerate}
      \end{definition}
      \begin{definition}
        A generalised random field is a continuous linear operator $\eta:D\to L^2(\Omega)$.
      \end{definition}

      With a slight abuse we use the following notation
      \begin{equation*}
        \eta(\phi_1,\ldots,\phi_k) = (\eta(\phi_1),\ldots,\eta(\phi_k)).
      \end{equation*}
      For two generalised random fields we write $\eta \= \xi$
      and say that $\eta$ and $\xi$ are equivalent in finite-dimensional distributions if for every $k\in\N$ and every vector of test functions $(\phi_1,\ldots\phi_k) \in D^k$ we have
      \begin{equation*}
        \eta(\phi_1,\ldots,\phi_k) \= \xi(\phi_1,\ldots,\phi_k).
      \end{equation*}
      Similarly we say $\eta_n$ converge \emph{weakly in finite-dimensional distributions} to $\eta$ if for every vector of test functions $(\phi_1,\ldots\phi_k) \in D^k$
      \begin{equation*}
        \wlim_{n\to\infty}\eta_n(\phi_1,\ldots,\phi_k) \= \eta(\phi_1,\ldots,\phi_k),
      \end{equation*}
      and write $\wlim_{n\to\infty}\eta_n\=\eta$.
    \end{subsection}
    \begin{subsection}{Random measures and integrals}
      \begin{definition}
        Let $\sigma$ be a Borel measure on $\X$ and $\Pi\subset\Bor(\X)$ be a family of all sets $A$ such that $\sigma(A)$ is finite.
        An orthogonal random measure $Z$ with a reference measure $\sigma$ is a function $Z:\Pi\to L^2(\Omega;\C)$ such that
        \begin{enumerate}
          \item $EZ(A)=0$,
          \item $EZ(A_1)\overline{Z(A_2)} = \sigma(A_1\cap A_2)$,
          \item $Z(\bigcup_{n=1}^\infty A_n) = \sum_{n=1}^\infty Z(A_n)$ {\rm(}as a limit in $L^2(\Omega;\C)${\rm)}
          for every pairwise\ \smallskip disjoint collection of sets $\{A_n\}\subset\Pi$ such that $\bigcup A_n\in\Pi$.
        \end{enumerate}
      \end{definition}
      Notice that, despite the name, in general $Z$ is \emph{not} a measure since $\Pi$ is not necessarily a sigma-algebra.
      It may even be the case that $Z$ cannot be extended to a proper measure. It is however possible to define an integral with respect to $Z$ (see
      e.g.~\cite[Ch.~2,~Sec.~3]{MR1885884} or \cite[Ch.~5,~Sec.~3]{MR0247660} for more details).
      Here, let us only briefly recall that the properties of an orthogonal random measure allow us to define a unique linear operator $I_Z: L^2(\X,\Bor(\X),\sigma;\C)\to L^2(\Omega;\C)$,
      such that
      \begin{equation}\label{simple-functions}
        I_Z\Big(\sum_{i\leq n} c_i \mathbbm{1}_{A_i}\Big) = \sum_{i\leq n} c_i Z(A_i)\ \text{a.e. for all $n\in \N$, $A_i\in \Pi$ and $c_i\in\C$,}
      \end{equation}
      and
      \begin{equation*}
        E|I_Z(f)|^2 = \int_\X |f|^2\,d\sigma\quad\text{for all $f\in L^2(\X,\Bor(\X),\sigma;\C)$.}
      \end{equation*}
      \begin{rem}\label{pisystem}
        In order to define an othogonal random measure it suffices to consider a family of sets $\Pi_0$ such that $A_1,A_2\in\Pi_0$
        implies $A_1\cap A_2\in\Pi_0$ (a $\pi$-system) and there exist $A_1\subset A_2\subset \ldots\in \Pi_0$
        such that $\bigcup_{n=1}^\infty A_n = \X$ (see~\cite[Ch.~2,~Sec.~3,~Thm.~19]{MR1885884}).
      \end{rem}

      We say that $Z_n$ converges \emph{weakly in finite-dimensional distributions} to $Z$
      if for every vector of functions $(f_1,\ldots,f_k) \in L^2(\X,\Bor(\X),\sigma;\C)^k$
      we have
      \begin{equation*}
        \wlim_{n\to\infty}(I_{Z_n}(f_1),\ldots,I_{Z_n}(f_k)) \= (I_Z(f_1),\ldots, I_Z(f_k)),
      \end{equation*}
      denoting $\wlim_{n\to\infty}Z_n \= Z$.
      \begin{prop}\label{ccw-measures}
        Suppose we have
        \begin{equation*}
          \lim_{n\to\infty}E|Z_n(A)-Z(A)|^2 = 0\quad\text{for every $A\in\Pi$}.
        \end{equation*}
        Then $\wlim_{n\to\infty}Z_n\=Z$.
      \end{prop}
      \begin{proof}
        Because of~\rfb{simple-functions} and the continuity of the operators $I_{Z_n}$ and $I_Z$, we have
        \begin{equation*}
          \lim_{n\to\infty} E|I_{Z_n}(f)-I_{Z}(f)|^2 = 0\quad\text{for every $f\in L^2(\X,\Bor(\X),\sigma;\C)$}.
        \end{equation*}
        Denote
        \begin{equation*}
          k_n = E|I_{Z_n}(f)-I_{Z}(f)|^2.
        \end{equation*}
        Then by the Chebyshev inequality we have
        \begin{equation*}
          P(|I_{Z_n}(f)-I_{Z}(f)|>\varepsilon) \leq \frac{k_n}{\varepsilon^2},
        \end{equation*}
        which shows that for every $f\in L^2(\X,\Bor(\X),\sigma;\C)$, $I_{Z_n}(f)$ converges in probability to $I_Z(f)$ and therefore also
        \begin{equation*}
          \wlim_{n\to\infty}I_{Z_n}(f)\=I_Z(f).
        \end{equation*}
        Then by linearity of the operators $I_{Z_n}$ and $I_Z$ and the Cramér--Wold theorem (see~\cite[Thm.~29.4]{MR1324786}) we obtain
          $\wlim_{n\to\infty}Z_n\=Z$.
      \end{proof}
      From now on we shall write
      \begin{equation*}
        I_Z(f) = \int_\X f(x)\,Z(dx).
      \end{equation*}
    \end{subsection}
    \begin{subsection}{Broad-sense stationary random fields and $H$ space}
      For $\phi\in D$ and a parameter $h\in \X$ we denote the translation operator by
      \begin{equation*}
        \tau_h\phi(x) = \phi(x+h).
      \end{equation*}
      \begin{definition}
        A generalised random field $\eta$ is broad-sense stationary if for every $\phi,\psi\in D$
        \begin{align*}
          &E\eta(\phi) = E\eta(\psi),\\
          &E\eta(\phi)\eta(\psi) = E\eta(\tau_h\phi)\eta(\tau_h\psi)\quad\text{for every $h\in\X$}.
        \end{align*}
      \end{definition}
      In this work we limit ourselves to broad-sense stationary random fields with zero mean, i.e. $E\eta(\phi) = 0$ for every $\phi\in D$.

      We recall two basic theorems from the theory of broad-sense stationary random fields, the proofs of which may be found
      in~\cite{MR0173945} or~\cite{MR0247660}.
      A Borel measure $\sigma$ such that
      \begin{equation*}
        \int_\X\frac{1}{(1+|x|^2)^p}\,d\sigma
      \end{equation*}
      is finite for some $p>0$ is called tempered.

      \begin{theorem}
        Let $\eta$ be a broad-sense stationary random field. There exists a tempered measure $\sigma$, such that
        \begin{equation*}
          E\eta(\phi)\eta(\psi) = \int_\X\FT{\psi}\overline{\FT{\phi}}\,d\sigma.
        \end{equation*}
        The measure $\sigma$ is called the spectral measure of the random field $\eta$.
      \end{theorem}
      \begin{theorem}\label{karhunen}
        Let $\eta$ be a broad-sense stationary random field with the spectral measure $\sigma$.
        There exists an orthogonal random measure $Z$ with the reference measure $\sigma$ such that for every $\phi\in D$
        \begin{equation}\label{spectral}
          \eta(\phi) = \int_\X\FT{\phi}\,dZ.
        \end{equation}
      \end{theorem}
      Conversely, given an orthogonal random measure $Z$, by the means of formula~\rfb{spectral}, we may construct a generalised broad-sense stationary random field.

      Using the representation from Theorem~\ref{karhunen} we may also consider the following continuous extension of a broad-sense stationary random field $\eta$ to
      a space larger than~$D$.

      Given the spectral measure $\sigma$ associated with $\eta$, we consider the space $D$ as a pre-Hilbert space with the inner product
      \begin{equation}\label{scalar-product}
        (f,g) = \int_\X\FT{f}\overline{\FT{g}}\,d\sigma + \int_\X fg\,dx.
      \end{equation}
      We define the space $\F{\sigma}$ as the completion of $D$ with respect to \rfb{scalar-product}.
      Then $\F{\sigma}$ is a Hilbert space with the norm $\|f\|_{\F{\sigma}}^2 = (f,f)$.
      Furthermore, $D$ is a dense linear subspace of $\F{\sigma}$ and thus $\F{\sigma}$ is a dense linear subspace of $L^2(\X,\Bor(\X),dx)$.

      We may now define an extension $\eta: \F{\sigma} \to L^2(\Omega)$ by the formula
      \begin{equation*}
        \eta(f) = \int_\X\FT{f}\,dZ.
      \end{equation*}
      Since for every $f\in \F{\sigma}$ we have
      \begin{equation*}
        E(\eta(f))^2 = E\Big|\int_\X\FT{f}\,dZ\Big|^2 = \int_\X |\FT{f}|^2\,d\sigma \leq \|f\|_{\F{\sigma}}^2 < \infty,
      \end{equation*}
      this extension is well-defined and continuous.
      \begin{rem}\label{separability}
        $\F{\sigma}$ is a separable space.
      \end{rem}
    \end{subsection}
    \begin{subsection}{Transformations of stationary random fields}
      Consider a broad-sense stationary random field $\eta$ with spectral measure $\sigma$, extended to the space $\F{\sigma}$
      as described in the previous section.
      Let $\A : L^2(dx)\to L^2(dx)$ be a continuous linear operator. Its standard adjoint operator
      $\A^*$ is determined by the formula
      \begin{equation*}
          \int_\X\A fg\,dx = \int_\X f\A^* g\,dx\quad\text{for all $f,g\in L^2(dx)$}
      \end{equation*}
      \begin{definition}\label{tbssrf}
        Let $\A : L^2(dx)\to L^2(dx)$ and assume $\F{\sigma}$ is an invariant subspace of $\A^*$.
        A field $\A\eta: \F{\sigma} \to L^2(\Omega)$ defined as
        \begin{equation*}
          \A\eta(\phi) = \eta(\A^*\phi)\quad\text{for every $\phi\in\F{\sigma}$}
        \end{equation*}
        is called a transformed broad-sense stationary random field.
      \end{definition}
      \begin{rem}
       Let us emphasise that the operator $\A^*$, in general, \emph{is not} the adjoint of $\A$ on $\F{\sigma}$.
      \end{rem}
      \begin{definition}\label{locint}
        Suppose $\eta_t$ are broad-sense stationary random fields with spectral measures $\sigma_t$ and consider a Banach space
        $X$ such that $X\subset \F{\sigma_t}$ for every $t\geq 0$.
        A time-dependent family $\A_t\eta_t$ of transformed broad-sense stationary random fields is called locally integrable
        with respect to $X$ if for every $\psi\in C_c(\T,X)$
        \begin{enumerate}
          \item $t\mapsto\A_t\eta_t(\psi(t))$ is a measurable function,
          \item $\intT{E(\A_t\eta_t(\psi(t)))^2} <\infty$.
        \end{enumerate}
        An immediate consequence of this definition is that the mapping $t\mapsto\A_t\eta_t(\psi(t))$ is Bochner-integrable
        (see~\cite[Ch.~2,~Sec.~2,~Thm.~2]{MR0453964}).
      \end{definition}
    \end{subsection}
    \begin{subsection}{Regular stationary random fields}\label{regular}
      Having defined \emph{generalised} random fields, we should also mention their non-generalised, regular counterparts.
      We restrict ourselves to the one-dimensional, broad-sense stationary case for simplicity of the exposition.
      \begin{definition}
        A continuous mapping $\eta:\R\to L^2(\Omega)$ is called broad-sense stationary if for every $x,y\in\R$
        \begin{equation*}
          E\eta(x)\eta(y) \= R(|x-y|),
        \end{equation*}
        for some $R:\R\to\R$.
      \end{definition}
      It is easy to see that $R$ must be continuous and positive-definite, thus the Bochner theorem asserts that there exists a finite measure $\sigma$ on $\R$, such that
      $R(x) = \int_\R e^{ix\xi}\,\sigma(d\xi)$. In full analogy to Theorem~\ref{karhunen} we also have the representation $\eta(x) = \int_\R e^{ix\xi}\,Z(d\xi)$,
      where $Z$ is an orthogonal random measure with reference measure $\sigma$.

      Since
      \begin{equation*}
        \int_\R |\FT{\phi}|^2\,d\sigma < \sigma(\R)(\lambda(\supp\phi)\sup|\phi|)^2 < \infty,
      \end{equation*}
      we may define the corresponding generalised broad-sense stationary random field $\eta^\circ$ by using the formula
      \begin{equation*}
        \eta^\circ(\phi) = \int_\R \FT{\phi}\,dZ.
      \end{equation*}
      Conversely, suppose we have a generalised broad-sense stationary random field $\eta^\circ$, with its associated orthogonal random measure $Z$ and a finite spectral measure~$\sigma$.
      Then $E(\int_\R e^{ix\xi}\,Z(d\xi))^2 = \sigma(\R)$ and we may define a regular broad-sense stationary random field $\eta$ by the formula
      \begin{equation*}
        \eta(x) = \int_\R e^{ix\xi}\,Z(d\xi).
      \end{equation*}
    \end{subsection}
    \begin{example}\label{white-noise}
      Consider the standard (two-sided) complex-valued Brownian motion $B_t$ on $\R$ and define $W_0([0,t]) = B_t$.
      We may then extend $W_0$ to an orthogonal random measure $W$ on $\R$, defined on the family of bounded sets
      (i.e. of finite Lebesgue measure, cf.~Remark~\ref{pisystem}). Indeed, we may also check that
      \begin{equation*}
      EW([0,t])\overline{W([0,s])} = EB_t\overline{B_s} = \min\{s,t\} = \lambda([0,t]\cap [0,s]),
      \end{equation*}
      which proves that $\lambda$ is in fact the reference measure of $W$.

      Define a generalised stationary random field $\eta(\phi) = \int_\R\FT{\phi}W(dx)$. Then,
      \begin{equation*}
        E\eta(\phi)\eta(\psi) = \int_\R\FT{\phi}\overline{\FT{\psi}}\,dx = (2\pi)^d\int_\R\phi\psi\,dx
      \end{equation*}
      (the latter equality follows from the Parseval identity). The field $\eta$ is called the (Gaussian) white-noise.
      It follows immediately that $\F{dx} = L^2(dx)$.
      Since $\lambda$ is not a finite measure, one may show that this random field cannot be represented as a regular random field.
    \end{example}
  \end{section}
  \begin{section}{Solutions to Cauchy problem}
    The theory recalled in the previous section allows us to define solutions to the initial value problem
    \begin{equation}\label{heat}
      \left\{
      \begin{aligned}
        &\dt u = \LL u\quad&\text{on $\T\times\X$}, \\
        &u(0) \= \eta_0\quad&\text{on $\X$}
      \end{aligned}
      \right.
    \end{equation}
    and to study their properties. We assume the following
    \begin{description}
      \item[(\ref{heat}a)] 
        $\eta_0$ is a generalised broad-sense stationary random field;
      \item[(\ref{heat}b)] 
        $\LL$ is a generator of a $C_0$-semigroup $\{e^{t\LL}\}_{t\geq0}$ of linear operators on $L^2(dx)$ such that
        $D\subset\Dom(\LL^*)$ and $\F{\sigma_0}$ is $\LL^*$-admissible (see~Definition~\ref{admissibility}).
    \end{description}
    \begin{rem}
      The operator $\LL^*$ is also a generator of a $C_0$-semigroup on $L^2(dx)$ and $(e^{t\LL})^* = e^{t\LL^*}$ (see~\cite[Ch.~1,~Cor.~10.6]{MR710486}).
    \end{rem}

    \begin{definition}\label{solutions}
      A time-dependent family $u(t)$ of transformed broad-sense stationary random fields, which is locally integrable
      with respect to $\F{\sigma_0}$ (see Definition~\ref{locint})
      is called a solution to problem \rfb{heat} if for every vector
      $(\psi_1,\ldots,\psi_k)\in \DT^k$ we have
      \begin{equation}\label{solutions-equation}
        \intT{u(t)\big(\B\psi_1(t),\ldots,\B\psi_k(t)\big)}\=-\eta_0\big(\psi_1(0),\ldots,\psi_k(0)\big).
      \end{equation}
    \end{definition}

    \begin{rem} Notice that the expression under the integral in~\rfb{solutions-equation} is well-defined since for every $\psi(t)\in \DT$ a direct calculation shows
      \begin{equation*}
        \B\psi(t)\in \CTL.
      \end{equation*}
    \end{rem}

    In the following theorems we extend the meaning of classical semigroup solutions to cases where the initial conditions
    are broad-sense stationary random fields.

    \begin{theorem}[Existence of solutions]\label{existence}
      Assume {\rm(\ref{heat}a)} and {\rm(\ref{heat}b)}.
      The family of transformed broad-sense stationary random fields $e^{t\LL}\eta_0$ (see Definition~\ref{tbssrf}) is a solution to problem~\rfb{heat}.
    \end{theorem}
    \begin{proof}
      First, we show that the function $t\mapsto\eL{t}\psi(t)$ is continuous for every $\psi\in \CTL$.
      Recall that there exist constants $M\geq1$ and $\omega\geq0$ such that $\|\eL{t}\|_{\F{\sigma_0}}\leq Me^{t\omega}$ (see \cite[Ch.~1,~Thm.~2.2]{MR710486}).
      Denote by $T$ the upper bound of the support of $\psi$ and let $t_n\to t$. Then
      \begin{equation*}
        \begin{split}
          \|\eL{t}\psi(t)& - \eL{t_n}\psi(t_n)\|_{\F{\sigma_0}}\\
          &\leq \|(\eL{t}-\eL{t_n})\psi(t)\|_{\F{\sigma_0}}+\|\eL{t_n}(\psi(t)-\psi(t_n))\|_{\F{\sigma_0}}\\
          &\leq \|(\eL{t}-\eL{t_n})\psi(t)\|_{\F{\sigma_0}}+Me^{T\omega}\|\psi(t)-\psi(t_n)\|_{\F{\sigma_0}}.
        \end{split}
      \end{equation*}
      It follows from the continuity of $\psi$ that $\lim_{n\to\infty}\|\psi(t)-\psi(t_n)\|_{\F{\sigma_0}} = 0$. Since the semigroup is strongly continuous,
      the function $t\mapsto\eL{t}\phi$ is continuous for every $\phi\in \F{\sigma_0}$ (see \cite[Ch.~1,~Cor.~2.3]{MR710486}) and therefore
      \begin{equation*}
        \lim_{n\to\infty}\|(\eL{t}-\eL{t_n})\psi(t)\|_{\F{\sigma_0}} = 0.
      \end{equation*}
      Because $\eta_0$ is a continuous operator, the function $e^{t\LL}\eta_0(\psi(t)) = \eta_0(\eL{t}\psi(t))$ is also continuous
      and hence measurable. Now,
      \begin{multline*}
        \intT{E(e^{t\LL}\eta_0(\psi(t)))^2} = \intT{E(\eta_0(\eL{t}\psi(t)))^2} \\= \intT{\int_\X|\FT{\eL{t}\psi(t)}|^2\,d\sigma_0} \leq T\cdot\sup_{t\in\Tb}\|\eL{t}\psi(t)\|_{\F{\sigma_0}}^2 < \infty.
      \end{multline*}
      This completes the proof of the local integrability of $e^{t\LL}\eta_0$  with respect to $\F{\sigma_0}$ (see Definition~\ref{locint}).
      Because $\eL{t}$ is a $C_0$-semigroup on $\F{\sigma_0}$, a direct calculation shows
      \begin{equation*}
        \eL{t}\B\psi(t) = \dt\big(\eL{t}\psi(t)\big)\quad\text{for every $\psi(t)\in \DT$}.
      \end{equation*}
      This relation together with the fact that $\B\psi(t)\in \CTL$ mean that both $\eL{t}\B\psi(t)$ and $\eta_0\big(\dt\big(\eL{t}\psi(t)\big)\big)$
      are integrable with respect~to~$t$. Finally, we apply the Hille theorem (see \cite[Ch.~2,~Thm.~6]{MR0453964}) to obtain
      \begin{equation*}
        \intT{\eta_0\big(\dt\big(\eL{t}\psi(t)\big)\big)} = \eta_0\Big(\intT{\dt(\eL{t}\psi(t))}\Big) = \eta_0(-\psi(0)).
      \end{equation*}
      Hence $e^{t\LL}\eta_0$ satisfies the conditions of Definition~\ref{solutions}.
    \end{proof}
    Our next goal is to show that a solution to problem~\rfb{heat} is unique under a suitable assumption of continuity.
    \begin{definition}\label{def-strong-continuity}
      Let $\eta$ be a broad-sense stationary random field with spectral measure $\sigma$. A family $\A_t\eta$ of transformed broad-sense stationary random fields
      is called strongly continuous if for every $\phi\in \F{\sigma}$ the mapping $t\mapsto \A^*_t \phi$ is continuous.
    \end{definition}
    \begin{rem}\label{strong-continuity-of-semigroup}
    It follows directly from the above definition, that for every broad-sense stationary random field $\eta$,
    if $\A_t$ is a $C_0$-semigroup on $\F{\sigma}$, then $\A_t\eta$ is strongly continuous.
    \end{rem}
    \begin{definition}
      We say that $\psi_n\to\psi$ in $\CTL$ if there exists a compact set $K$ such that $\bigcup_{n}\supp\psi_n\cup\supp\psi\subset K$ and
      \begin{equation*}
        \lim_{n\to\infty}\sup\|\psi_n(t)-\psi(t)\|_{\F{\sigma_0}}=0.
      \end{equation*}
    \end{definition}
    \begin{lemma}\label{strong-continuity}
      If $\psi_n\to\psi$ in $\CTL$ and $u(t)$ is a strongly continuous family of random fields, then
      \begin{equation*}
        \lim_{n\to\infty}\sup_{t\in\T}\|u(t)(\psi_n(t)-\psi(t))\|_{L^2(\Omega)} = 0.
      \end{equation*}
    \end{lemma}
    \begin{proof}
      Denote $u(t) = \A_t\eta_0$.
      Since $t \mapsto \A^*_t\phi$ is continuous for every $\phi$, it is bounded on compact sets.
      Thus, from the Banach-Steinhaus theorem applied to the family $\A^*_t$, we obtain that for every compact set $K$
      \begin{equation*}
        \sup_{t\in K}\|\A^*_t\|_{\F{\sigma_0}\to \F{\sigma_0}} = \sup\{\|\A^*_t\phi\|_{\F{\sigma_0}}\,:\,\|\phi\|_{\F{\sigma_0}}=1\} = C_K < \infty.
      \end{equation*}
      Therefore
      \begin{align*}
        \lim_{n\to\infty}&\sup_{t\in\T}\|u(t)(\psi_n(t)-\psi(t))\|_{L^2(\Omega)}\\
        &\leq \lim_{n\to\infty} \sup_{t\in\T}\|\A_t^*(\psi_n(t)-\psi(t))\|_{\F{\sigma_0}}\\
        &\leq C_K \lim_{n\to\infty} \sup_{t\in K}\|\psi_n(t)-\psi(t)\|_{\F{\sigma_0}} = 0,
      \end{align*}
      where $K\supset\bigcup_{n}\supp\psi_n\cup\supp\psi$.
    \end{proof}
    \begin{lemma}\label{integral}
      Let $u_1(t)$ and  $u_2(t)$ be strongly continuous families of random fields.
      If for every vector $(\psi^1,\ldots,\psi^k)\in \CTL^k$ we have
      \begin{equation}\label{inteq}
        \intT{u_1(t)\big(\psi^1(t),\ldots,\psi^k(t)\big)}\=\intT{u_2(t)\big(\psi^1(t),\ldots,\psi^k(t)\big)},
      \end{equation}
      then $u_1(t)\=u_2(t)$.
    \end{lemma}
    \begin{proof}
      First we fix a vector $(\phi^1,\ldots,\phi^k)\in D^k$ and a point $s\in\T$.
      For every $j=1,\ldots,k$, let us choose a sequence $\psi^j_{n,1}(t)\in \CTL$ such that
      \begin{equation*}
        \lim_{n\to\infty}\psi^j_{n,1}(t)=\left\{\begin{aligned}&\phi^j,\quad t\in[s,s+h]\\&0\end{aligned}\right.
      \end{equation*}
      Define $\psi^j_{n,h}(t) = \psi^j_{n,1}(\frac{1}{h}(t-s)+s)$. By the Lebesgue dominated convergence theorem we then obtain for $i=1,2$
      \begin{equation*}
        \lim_{n\to\infty}\intT{u_i(t)(\psi^1_{n,h}(t),\ldots,\psi^k_{n,h}(t))} = \int_s^{s+h}{u_i(s)(\phi^1,\ldots,\phi^k)}\,dt.
      \end{equation*}
      Therefore by \rfb{inteq} we also have
      \begin{equation*}
        \frac{1}{h}\int_s^{s+h}{u_1(t)(\phi^1,\ldots,\phi^k)}\,dt \= \frac{1}{h}\int_s^{s+h}{u_2(s)(\phi^1,\ldots,\phi^k)}\,dt.
      \end{equation*}
      By strong continuity of $u_1$ and $u_2$, when we pass to the limit with $h\to0$ on both sides of the above expression, we obtain
      \begin{equation*}
        u_1(s)(\phi^1,\ldots,\phi^k)\=u_2(s)(\phi^1,\ldots,\phi^k),
      \end{equation*}
      which ends the proof, since $\phi^1,\ldots,\phi^k$ and $s$ were arbitrary.
    \end{proof}
    \begin{theorem}[Uniqueness of solutions]\label{uniqueness}
      Assume {\rm(\ref{heat}a)} and {\rm(\ref{heat}b)}.
      If $u_1$ and $u_2$ are two strongly continuous solutions to problem \rfb{heat}, then $u_1(t) \= u_2(t)$ for every $t\geq 0$.
    \end{theorem}
    \begin{proof}
      Since the space $\F{\sigma_0}$ is separable, by Theorem~\ref{separable-density} in Appendix, the image $\B(\DT)$ is dense in $\CTL$.
      Thus for every $\psi\in \CTL$ we may find a sequence
      \begin{equation*}
      \psi_n=\B\phi_n,\quad\text{for some $\phi_n\in \DT$},
      \end{equation*}
      such that $\psi_n\to\psi$ in $\CTL$.

      Then by Lemma~\ref{strong-continuity},
      \begin{equation*}
        \lim_{n\to\infty}\sup_{t\in\T}\|u_i(t)(\psi_n(t)-\psi(t))\|_{L^2(\Omega)} = 0,\quad\text{$i=1,2$}.
      \end{equation*}
      Thus
      \begin{multline}\label{limit-u}
        \lim_{n\to\infty}\intT{\|u_i(t)(\psi_n(t)-\psi(t))\|_{L^2(\Omega)}}
        \\\leq \lambda(K)\lim_{n\to\infty}\sup_{t\in\T}\|u_i(t)(\psi_n(t)-\psi(t))\|_{L^2(\Omega)} = 0,\quad\text{$i=1,2$},
      \end{multline}
      where $K\supset\bigcup_{n}\supp\psi_n\cup\supp\psi$.

      Let us fix a vector $(\psi^1,\ldots,\psi^k)\in \CTL^k$. For each of its elements we find a sequence
      $(\phi^i_n)_{n=1}^\infty\subset \DT$ such that $\B\phi^i_n \to \psi^i$ in $\CTL$. Because $u_1$ and $u_2$ are both solutions to problem~\rfb{heat},
      for every $n$ we have
      \begin{multline*}
        \intT{u_1(t)\big(\B\phi^1_n(t),\ldots,\B\phi^k_n(t)\big)}  \\ \= \intT{u_2(t)\big(\B\phi^1_n(t),\ldots,\B\phi^k_n(t)\big)}.
      \end{multline*}
      Using \rfb{limit-u} we may pass to the limit on both sides of this equality in order to obtain
      \begin{equation*}
        \intT{u_1(t)\big(\psi^1(t),\ldots,\psi^k(t)\big)}  \= \intT{u_2(t)\big(\psi^1(t),\ldots,\psi^k(t)\big)}.
      \end{equation*}
      Then it follows from Lemma~\ref{integral} that $u_1(s) \= u_2(s)$.
    \end{proof}

    Let us conclude this section by summarizing the above results.
    \begin{rem}\label{one-solution}
    Assume {\rm(\ref{heat}a)} and {\rm(\ref{heat}b)}.
    Let a strongly continuous family of transformed broad-sense stationary random fields $u(t)$ be a solution to problem~\rfb{heat} in the sense of Definition~\ref{solutions}.
    Then Theorems \ref{existence} and \ref{uniqueness}, as well as~Remark~\ref{strong-continuity-of-semigroup}, imply
    \begin{equation*}
     u(t) \= e^{t\mathcal{L}}\eta_0,
    \end{equation*}
    i.e.~the semigroup solution $e^{t\LL}\eta_0$ is the unique strongly continuous solution of problem~\rfb{heat},
    up to the equivalence in finite-dimensional distributions.
    \end{rem}
  \end{section}
  \begin{section}{Scaling limits of semigroup solutions}
    In this section we study certain scaling properties of the semigroup solution to Cauchy problem~\rfb{heat}, which is unique
    under conditions outlined in Remark~\ref{one-solution}.

    To simplify the notation we introduce a family of scaling transformations $\nu_r$ defined for every $r>0$ by
    \begin{align*}
      &\nu_r\phi(x) = r^d\phi(rx)\quad\text{for every $\phi\in L^2(dx)$},\\
      &\nu_r Z(A) = Z(r A)\quad\text{for every orthogonal random measure $Z$}.
    \end{align*}
    We use the same symbol for different scalings of functions and measures. This is done in order to obtain the following
    identity
    \begin{equation}\label{fourier-scaling}
      \int_\X\FT{\nu_r \phi}\,dZ = \int_\X\FT{\phi}\,d(\nu_rZ),
    \end{equation}
    which is an immediate consequence of change of variables in the definition of the Fourier transform.
    For a generalised random field $\xi$ we write $\xi\nu_r$ to mean $\xi\nu_r(\phi) = \xi(\nu_r \phi)$ for every $\phi\in D$.
    \begin{theorem}[Scaling limit]\label{scaling}
      Let $Z_0$ be the orthogonal random measure associated with the initial condition $\eta_0$ to problem~\rfb{heat} and $\sigma_0$ be its
      spectral measure.
      Suppose there are constants $\alpha$,~$\gamma\in\R$ and an orthogonal random measure $Z$ with a reference measure $\sigma$ such that
      \begin{equation}\label{measure-scaling}
        \wlim_{T\to\infty}T^{\gamma} \nu_{T^\alpha} Z_0 \= Z.
      \end{equation}
      Suppose further, there is a constant $\beta\in\R$ and an operator $\P$ satisfying {\rm(\ref{heat}b)} and such that $\F{\sigma}$ is also $\P^*$-admissible, and
      \begin{equation}\label{semigroup-scaling}
        \lim_{T\to\infty}T^{2\gamma}\int_\X|\FT{\eL{tT}(\phiT)- \nu_{T^\alpha} (e^{t\P^*}\phi)}|^2\,d\sigma_0 = 0\quad\text{for every $\phi\in D$}.
      \end{equation}
      Then for the broad-sense stationary random field $\eta(\phi) = \int_\X\FT{\phi}\,dZ$ we have
      \begin{equation*}
        \wlim_{T\to\infty}T^\gamma (e^{Tt\LL}\eta_0)\nu_{T^\beta} \= e^{t\P}\eta.
      \end{equation*}
    \end{theorem}
    \begin{proof}
      For every $\phi\in D$ we have the representation (see Theorem~\ref{karhunen})
      \begin{equation*}
        e^{Tt\LL}\eta_0(\phiT) = \int_\X\FT{\eL{tT}(\phiT)}\,dZ_0.
      \end{equation*}
      We rewrite this expression as the following sum
      \begin{multline*}
        \int_\X\FT{\eL{tT}(\phiT)}\,dZ_0 = \\
        = \int_\X\FT{\eL{tT}(\phiT) -  \nu_{T^\alpha} (e^{t\P^*}\phi)}\,dZ_0 + \int_\X\FT{\nu_{T^\alpha} (e^{t\P^*}\phi)}\,dZ_0.
      \end{multline*}
      For the first term on the right-hand side we use assumption~\rfb{semigroup-scaling} in order to obtain
      \begin{multline*}
        \lim_{T\to\infty} E\Big|T^\gamma\int_\X\FT{\eL{tT}(\phiT) - \nu_{T^\alpha}(e^{t\P^*}\phi)}\,dZ_0\Big|^2 \\
        = \lim_{T\to\infty}T^{2\gamma}\int_\X|\FT{\eL{tT}(\phiT) - \nu_{T^\alpha}(e^{t\P^*}\phi)}|^2\,d\sigma_0 = 0.
      \end{multline*}
      For the second term, because of identity~\rfb{fourier-scaling}, we have
      \begin{equation*}
        \int_\X \FT{\nu_{T^\alpha}(e^{t\P^*}\phi)}\,dZ_0 = \int_\X\FT{e^{t\P^*}\phi}\,d(\nu_{T^\alpha} Z_0).
      \end{equation*}
      Let us fix a vector $(\phi^1,\ldots,\phi^k)\in D^k$. It follows from assumption~\rfb{measure-scaling} that
      \begin{multline*}
        \wlim_{T\to\infty}T^{\gamma} \int_\X \big(\FT{e^{t\P^*}\phi^1},\ldots,\FT{e^{t\P^*}\phi^k}\big)\,d(\nu_{T^\alpha} Z_0) \\
        \= \int_\X \big(\FT{e^{t\P^*}\phi^1},\ldots,\FT{e^{t\P^*}\phi^k}\big)\,dZ = e^{t\P}\eta(\phi^1,\ldots,\phi^k)
      \end{multline*}
      and finally
      \begin{equation*}
        \wlim_{T\to\infty}T^\gamma (e^{Tt\LL}\eta_0)\nu_{T^\beta} \= e^{t\P}\eta.\qedhere
      \end{equation*}
    \end{proof}
    \begin{rem}\label{trivial-scaling}
      Notice that if for a suitable choice of $\alpha$, $\beta$, $\gamma$ either
      \begin{equation*}
        T^{\gamma} \nu_{T^\alpha} Z_0 \= Z_0
     \end{equation*}
     or
     \begin{equation*}
       \eL{tT}(\phiT) = \nu_{T^\alpha}(\eL{t}\phi)\quad\text{for every $\phi\in D$},
     \end{equation*}
     then one of the hypotheses \rfb{measure-scaling} and \rfb{semigroup-scaling} in Theorem~\ref{scaling} is trivially satisfied.
    \end{rem}
    \begin{rem}
      Observe that in the context of Theorem~\ref{scaling}
      the limit family of random fields $e^{t\P}\eta$ is a solution to the following problem
      \begin{equation*}
        \left\{
        \begin{aligned}
          &\dt u = \P u\quad&\text{on $\T\times\X$}, \\
          &u(0) \= \eta\quad&\text{on $\X$}.
        \end{aligned}
        \right.
      \end{equation*}
    \end{rem}
  \end{section}
  \begin{section}{Examples}
    \begin{subsection}{Heat equation with Gaussian noise}\label{heat-ex}
      First, we illustrate and explain our results in the simplest case of the heat equation
      \begin{equation}\label{heat2}
        \left\{
        \begin{aligned}
          &\dt u = \Delta u\quad&\text{on $\T\times\R$}, \\
          &u(0) \= \eta_0\quad&\text{on $\R$}.
        \end{aligned}
        \right.
      \end{equation}
      Here, $\Delta=\partial_x^2$ denotes the Laplace operator, or more precisely, its closure in $L^2(dx)$ with the domain $\Dom(\Delta) = W^{2,2}(\R)$ (the usual Sobolev space).
      We introduce the orthogonal random measure $Z_0$ associated with $\eta_0$ by the means of the formula (see Theorem~\ref{karhunen})
      \begin{equation*}
        \eta_0(\phi) = \int_\R\FT{\phi}\,Z_0(dx)\quad\text{for all $\phi\in D$}.
      \end{equation*}
      \begin{theorem}\label{heat-solution}
        Keeping the above notation, assume there exists a function $f>0$ such that
        \begin{equation}\label{spectrum}
          Z_0(A) = \int_A\sqrt{f(x)}\,W(dx),
        \end{equation}
        where $W$ is the white-noise orthogonal random measure (see Example~\ref{white-noise}).
        Then the family of transformed broad-sense stationary random fields $e^{t\Delta}\eta_0$ is the unique strongly continuous solution
        to problem~\rfb{heat2}, up to the equivalence in finite-di\-men\-sio\-nal distributions.
      \end{theorem}
      \begin{proof}
        First we notice that $\Delta^* = \Delta$, $D\subset\Dom(\Delta)$ and $e^{t\Delta}$ is a $C_0$-semigroup of contractions on $L^2(dx)$.
        Consider the space $\F{\sigma_0} = \F{f(x)dx}$. Observe that $e^{-tx^2} \leq 1$ for $t\geq 0$,
        and therefore
        \begin{equation}\label{delta-invariant}
          \begin{split}
            \|e^{t\Delta}\phi\|_{\F{\sigma_0}}^2 &= \int_\R|\FT{e^{t\Delta}\phi}|^2f(x)\,dx + \int_\R|e^{t\Delta}\phi|^2\,dx \\
            &= \int_\R|e^{-tx^2}|^2|\FT{\phi}|^2f(x)\,dx  + \frac{1}{2\pi}\int_\R|e^{-tx^2}|^2||\FT{\phi}|^2\,dx \\
            &\leq \int_\R|\FT{\phi}|^2f(x)\,dx + \int_\R|\phi|^2\,dx \leq \|\phi\|_{\F{\sigma_0}}^2.
          \end{split}
        \end{equation}
        Thus $\F{\sigma_0}$ is $e^{t\Delta}$-invariant subspace of $L^2(dx)$. A similar calculation leads us to
        \begin{equation}\label{delta-strong-cont}
          \begin{split}
            \|e^{t\Delta}\phi-\phi\|_{\F{\sigma_0}}^2 = \int_\R|\FT{e^{t\Delta}\phi-\phi}|^2f(x)\,dx + \int_\R|e^{t\Delta}\phi-\phi|^2\,dx \\
            = \int_\R|e^{-tx^2}-1|^2|\FT{\phi}|^2f(x)\,dx  + \int_\R|e^{t\Delta}\phi-\phi|^2\,dx.
          \end{split}
        \end{equation}
        Using the Lebesgue dominated convergence theorem we obtain
        \begin{equation*}
          \lim_{t\to 0^+}\int_\R|e^{-tx^2}-1|^2|\FT{\phi}|^2f(x)\,dx  + \int_\R|e^{t\Delta}\phi-\phi|^2\,dx = 0,
        \end{equation*}
        which shows that $e^{t\Delta}$ is a $C_0$-semigroup on $\F{\sigma_0}$. Therefore $\F{\sigma_0}$ is $\Delta$-admis\-sible (cf. (\ref{heat}b)).

        Then it follows from Remark~\ref{one-solution} that the family $e^{t\Delta}\eta_0$ is the unique strongly continuous solution to problem~\rfb{heat2},
        up to the equivalence in finite-dimensio\-nal distributions.
      \end{proof}
      \begin{theorem}\label{heat-scaling}
        Keeping the above notation, assume additionally
        \begin{equation}\label{limit-bound}
          \begin{split}
            &\lim_{T\to\infty} T^{-k/2} f(x/\sqrt{T})|x|^k = 1 \quad\text{for some $k\in[0,1)$},\\
            &T^{-k/2} f(x/\sqrt{T}) \leq \frac{c}{|x|^k}\quad\text{for every $x$ and $T$ and some constant $c\in\R$}
          \end{split}
        \end{equation}
        and suppose $u$ is a strongly continuous solution to problem \rfb{heat2}.
        Then for the rescaled family of random fields
        \begin{equation*}
          u^T(t,\phi) = T^{\frac{1-k}{4}} u(Tt)(\nu_{T^{-1/2}}\phi)
        \end{equation*}
        and the family
        \begin{equation*}
          w(t,\phi) =\int_\R\frac{e^{-tx^2}\FT{\phi}}{|x|^{k/2}}\,W(dx)
        \end{equation*}
        we have $\wlim_{T\to\infty} u^T(t) \= w(t)$.
      \end{theorem}
      \begin{proof}
        By Theorem~\ref{heat-solution}, we have $u(t) \= e^{t\LL}\eta_0$. Let us observe that for every $\phi\in D$
        \begin{align*}
          e^{tT\Delta}(&\nu_{T^{-1/2}} \phi)(x)
          =\int_\R \frac{1}{\sqrt{4\pi tT}}\exp\Big(\frac{|x-y|^2}{4Tt}\Big)\frac{1}{\sqrt{T}}\phi(y/\sqrt{T})\,dy \\
          &=\int_\R \frac{1}{\sqrt{4\pi tT}}\exp\Big(\frac{|x-\sqrt{T}y|^2}{4Tt}\Big)\phi(y)\,dy \\
          &=\frac{1}{\sqrt{T}}\int_\R \frac{1}{\sqrt{4\pi t}}\exp\Big(\frac{|x/\sqrt{T}-y|^2}{4t}\Big)\phi(y)\,dy
          = \nu_{T^{-1/2}}(e^{t\Delta}\phi)(x).
        \end{align*}
        Thus $e^{t\Delta}\eta_0$ trivially satisfies assumption~\rfb{semigroup-scaling} of Theorem~\ref{scaling} with $\alpha=\beta=-\frac{1}{2}$ (see Remark~\ref{trivial-scaling}).
        We also have
        \begin{multline*}
          T^{\frac{1-k}{4}}Z_0(T^{-1/2}A)
          = T^{\frac{1-k}{4}}\int_\R\mathbbm{1}_{T^{-1/2}A}(x)\sqrt{f(x)}\,W(dx)\\
          \= T^{\frac{1-k}{4}}\int_A\sqrt{f(x/\sqrt{T})}\,W(dx/\sqrt{T})
          \= \int_A T^{-k/4}\sqrt{f(x/\sqrt{T})}\,W(dx),
        \end{multline*}
        where we used the scaling property of the white noise $W(rA) = \sqrt{r}W(A)$.
        Denote $Z(A) = \int_A |x|^{-k/2}\,W(dx)$. Then
        \begin{multline*}
          E|T^{\frac{1-k}{4}}Z_0(T^{-1/2}A) - Z(A)|^2
          = E\Big|\int_{A} T^{-k/4}f(x/\sqrt{T}) - \frac{1}{|x|^{k/2}}\,W(dx)\big|^2
          \\= \int_{A}\frac{1}{|x|^k}|T^{-k/4}\sqrt{f(x/\sqrt{T})|x|^k} - 1|^2\,dx.
        \end{multline*}
        Hence, using assumptions~\rfb{limit-bound} and by the Lebesgue dominated convergence theorem, we obtain
        \begin{equation*}
          \lim_{T\to\infty}E\big|T^{\frac{1-k}{4}}Z_0(T^{-1/2}A) - Z(A)\big|^2 = 0,
        \end{equation*}
        which because of Proposition~\ref{ccw-measures} leads us to
        \begin{equation*}
          \wlim_{T\to\infty}T^{\frac{1-k}{4}}\nu_{T^{-1/2}}Z_0 \= Z.
        \end{equation*}
        Finally, in view of Theorem~\ref{scaling} we have
        \begin{equation}\label{heat-scaling-limit}
          \wlim_{T\to\infty} u^T(t,\,\cdot\,) \= \int_\R \frac{e^{-tx^2}\FT{\,\cdot\,}}{|x|^{k/2}}W(dx).
        \end{equation}
      \end{proof}
      \begin{rem}
        The family obtained as the limit in the expression \rfb{heat-scaling-limit} is a solution to the problem
        \begin{equation}\label{limit-heat}
          \left\{
          \begin{aligned}
            &\dt u = \Delta u\quad&\text{on $\T\times\R$}, \\
            &u(0) \= \xi_0\quad&\text{on $\R$},
          \end{aligned}
          \right.
        \end{equation}
        where $\xi_0(\phi) = \int_\R\FT{\phi} \frac{1}{|x|^{k/2}}W(dx)$.
        It is worth noting that if the function $f$ in formula~\rfb{spectrum} is integrable, then the initial condition $\eta_0$ may be represented as a regular random field (see Subsection~\ref{regular}).
        Meanwhile, since $|x|^{-k}$ is not integrable, the same is never true for the initial condition $\xi_0$ of the limit problem~\rfb{limit-heat}. However, for $t>0$, the
        semigroup solution of~\rfb{limit-heat} may again be represented as a regular random field, since $|x|^{-k}e^{-2tx^2}$ is always integrable for any $k\in[0,1)$.
      \end{rem}
    \end{subsection}
    \begin{subsection}{Pseudo-differential operator with self-similar noise}
      In this example we consider the following problem
      \begin{equation}\label{heat3}
        \left\{
        \begin{aligned}
          &\dt u = \P u\quad&\text{on $\T\times\R$}, \\
          &u(0) \= \xi_0\quad&\text{on $\R$}.
        \end{aligned}
        \right.
      \end{equation}
      Here the initial condition has the form
      \begin{equation*}
        \xi_0(\phi) = \int_\R\FT{\phi} \frac{1}{|x|^{k/2}}W(dx),\quad\text{$k\in[0,1)$},
      \end{equation*}
      where $W$ is the Gaussian white-noise.
      The operator $\P$ is pseudo-differential, namely
      \begin{equation*}
       \FT{\P\phi} = p(x)\FT{\phi}
      \end{equation*}
      for a given continuous function $p:\R\to \mathbb{C}$ such that $\re p \leq 0$.

      \begin{theorem}
        The family of random fields $e^{t\P}\xi_0$ is the unique strongly continuous solution to problem~\rfb{heat3},
        up to the equivalence in finite-dimensional distributions.
      \end{theorem}
      \begin{proof}
        The proof is entirely analogous to the proof of Theorem~\ref{heat-solution}.

        The spectral measure of $\xi_0$ is $|x|^{-k}dx$.
        Consider the space $\F{|x|^{-k}dx}$. Observe that for $t\geq 0$ we have $|e^{tp(x)}| \leq 1$ and then by a similar
        calculation as in~\rfb{delta-invariant} and~\rfb{delta-strong-cont}
        we obtain that $\F{|x|^{-k}dx}$ is $\P^*$-admissible (cf. (\ref{heat}b)).

        Then it follows from Remark~\ref{one-solution} that the family $e^{t\P}\xi_0$ is the unique strongly continuous solution to problem~\rfb{heat3},
        up to the equivalence in finite-dimensional distributions.
      \end{proof}
      \begin{theorem}
        Assume there exists a continuous real function $q:\R\to \R$ and a parameter $\alpha$ such that
        \begin{equation}\label{multiplier-scaling}
          \lim_{T\to\infty} T p(T^\alpha x) = q(x)\quad\text{for every $x\in\R$}
        \end{equation}
        and suppose $u$ is a strongly continuous solution to problem \rfb{heat3}.
        Then for the rescaled family of random fields
        \begin{equation*}
          u^T(t,\phi) = T^{\frac{\alpha}{2}(k-1)}u(Tt)(\nu_{T^{\alpha}} \phi)
        \end{equation*}
        and the family
        \begin{equation*}
          w(t) = \int_\R\frac{e^{-tq(x)}\FT{\phi}}{|x|^{k/2}}\,W(dx)
        \end{equation*}
        we have $\wlim_{T\to\infty} u^T(t) \= w(t)$.
      \end{theorem}
      \begin{proof}
        Denote $Z_0(dx)= \frac{1}{|x|^{k/2}}W(dx)$ and $\sigma_0 = \frac{1}{|x|^k}dx$. Then
        \begin{align*}
          \nu_{T^{\alpha}}Z_0(A) =
          \int_\R\mathbbm{1}_{T^{\alpha}A}(x)\frac{1}{|x|^{k/2}}\,W(dx)
          &\= \int_\R\mathbbm{1}_{A}(x)\frac{1}{|T^{\alpha}x|^{k/2}}\,W(T^{\alpha}dx)\\
          &\= T^{-\frac{\alpha}{2}(k-1)}\int_\R\mathbbm{1}_{A}(x)\frac{1}{|x|^{k/2}}\,W(dx).
        \end{align*}
        Hence
        \begin{equation*}
          T^{\frac{\alpha}{2}(k-1)} \nu_{T^{\alpha}}Z_0(A)  \= Z_0(A)
        \end{equation*}
        and $Z_0$ satisfies assumption~\rfb{measure-scaling} of Theorem~\ref{scaling} with $\gamma = \frac{\alpha}{2}(k-1)$ (see Remark~\ref{trivial-scaling}).
        Define an operator $\Q$ by $\FT{\Q\phi}=q(x)\FT{\phi}$.
        Then we have
        \begin{multline*}
          \int_\R|\FT{e^{tT\P^*}(\nu_{T^\alpha}\phi) - \nu_{T^\alpha} (e^{t\Q^*}\phi)}|^2\,d\sigma_0\\
          = \int_\R \big|e^{tT\overline{p}(x)}\FT{\nu_{T^\alpha}\phi} -  e^{tq(T^{-\alpha} x)}\FT{\nu_{T^\alpha}\phi}\big|^2\,d\sigma_0 \\
          = \int_\R \big|e^{tT\overline{p}(T^\alpha x)} -  e^{tq(x)}\big|^2 \big|\FT{\phi}\big|^2\frac{T^\alpha}{|T^\alpha x|^k}\, dx\\
          = T^{-2\gamma}\int_\R \big|e^{tT\overline{p}(T^\alpha x)} -  e^{tq(x)}\big|^2 \big|\FT{\phi}\big|^2\frac{1}{|x|^k}\, dx
        \end{multline*}
        Therefore, because of~\rfb{multiplier-scaling} and since both $\re p$ and $q$ are non-positive, we may use Lebesgue dominated convergence theorem
        in order to obtain
        \begin{equation*}
        \lim_{T\to\infty}T^{2\gamma}\int_\R|\FT{e^{tT\P^*}(\nu_{T^\alpha}\phi) - \nu_{T^\alpha} (e^{t\Q^*}\phi)}|^2\,d\sigma_0 = 0.
        \end{equation*}
        This shows that assumption~\rfb{semigroup-scaling} is satisfied and it concludes the proof.
      \end{proof}
      \begin{rem}
        The family $w(t)$ obtained as the limit in the previous theorem is a solution to the problem
        \begin{equation*}
          \left\{
          \begin{aligned}
            &\dt u = \Q u\quad&\text{on $\T\times\R$}, \\
            &u(0) \= \xi_0\quad&\text{on $\R$}.
          \end{aligned}
          \right.
        \end{equation*}
        Moreover, it follows from assumption~\rfb{multiplier-scaling} that $\lim_{T\to\infty} T\overline{p}(T^\alpha x) = q(x)$ and therefore
        $q(x) = q(1)|x|^{-1/\alpha}$. Thus $\Q$ is in fact a fractional Laplace operator, namely for $s = -\frac{1}{2\alpha}$ and $c=-q(1)$ we have $\Q = -c(-\Delta)^s$.
      \end{rem}
    \end{subsection}
    \begin{subsection}{Heat equation with non-gaussian noise}
      In the last example we would like to present results similar to those obtained in~\cite{MR1632518}.
      Consider the following problem
      \begin{equation}\label{heat-lw}
       \left\{
        \begin{aligned}
         &\dt u = \Delta u\quad&\text{on $\T\times\R$}, \\
         &u(0) \= G(\eta_0)\quad&\text{on $\R$},
        \end{aligned}
       \right.
      \end{equation}
      where $\eta_0$ is a regular broad-sense stationary Gaussian random field with the orthogonal random measure $Z_0(A) =  \int_A\sqrt{f(x)}\,W(dx)$,
      and the function $G$ is such that $EG(\eta_0(x))^2$ is finite.
      Then the function $f$ is integrable, and the field $G(\eta_0)$ may be represented in the basis of Hermite polynomials (see~\cite[Thm.~1.1.2]{MR1344217})
      \begin{equation*}
        G(\eta_0(x)) \= \sum_{n=0}^\infty \frac{c_n}{n!} h_n(\eta_0(x)),
      \end{equation*}
      where
      \begin{equation*}
        h_n(\eta_0(x)) = \int_\R\cdots\int_\R e^{ix(y_1+\ldots+y_n)}\sqrt{f(y_1)\cdots f(y_n)}\,W(dy_1)\cdots W(dy_n)
      \end{equation*}
      and
      \begin{equation}\label{hermite}
        c_n = \frac{(-1)^n}{\sqrt{2\pi}}\int_\R G(x) \frac{d^n}{dx^n}e^{-x^2/2}\,dx
      \end{equation}
      (see~\cite[Prop.~1.1.4]{MR1344217}).
      This allows us to extend $G(\eta_0)$ to a generalised random field by the following expression (cf.~Subsection~\ref{regular})
      \begin{equation*}
        G(\eta_0)(\phi) = \sum_{n=0}^\infty \frac{c_n}{n!}\int_{R^n} \FT{\phi}(y_1+\ldots+y_n)\bigotimes_{i=1}^{n}\sqrt{f(y_i)}\,W(dy_i)
      \end{equation*}
      for every $\phi\in D$.
      \begin{prop}
      Each random field $h_n(\eta_0(x))$ is broad-sense stationary with an orthogonal random measure
      \begin{equation}\label{herm-rom}
        Z_n(A) = \int_{\R^n}\mathbbm{1}_{A}(y_1+\ldots+y_n)\bigotimes_{i=1}^{n}\sqrt{f(y_i)}\,W(dy_i).
      \end{equation}
      \end{prop}
      \begin{proof}
      We proceed by a direct calculation. Denote $\overline{y} = y_1+\ldots+y_n$.
      \begin{multline*}
        Eh_n(\eta_0(x_1+h))h_n(\eta_0(x_2+h)) \\
        = E\Big(\int_{\R^n} e^{i(x_1+h)\overline{y}}\bigotimes_{i=1}^{n}\sqrt{f(y_i)}\,W(dy_i)
                \int_{\R^n} e^{i(x_2+h)\overline{y}}\bigotimes_{i=1}^{n}\sqrt{f(y_i)}\,W(dy_i)\Big)\\
        = \int_{\R^n} e^{ih\overline{y}}e^{ix_1\overline{y}}e^{-ih\overline{y}}e^{-ix_2\overline{y}}\bigotimes_{i=1}^{n}f(y_i)\,dy_i
        = Eh_n(\eta_0(x_1))h_n(\eta_0(x_2)).\qedhere
      \end{multline*}
      \end{proof}
      \begin{rem}
        Formula~\rfb{herm-rom} determines the spectral measure $\sigma_n$ of the random field $h_n(\eta_0(x))$, namely
        \begin{equation*}
          \sigma_n(A) = \int_{\R^n}\mathbbm{1}_{A}(y_1+\ldots+y_n)f(y_1)\cdots f(y_n)\,dy_1\cdots dy_n = \int_A f^{*n}(y)\,dy.
        \end{equation*}
      \end{rem}
      Now we are in a position to identify the solution to problem~\rfb{heat-lw}.
      \begin{theorem}
        The family of random fields $e^{t\Delta}G(\eta_0)(\phi) = G(\eta_0)(e^{t\Delta}\phi)$ is the unique strongly continuous solution to problem~\rfb{heat-lw},
        up to the equivalence in finite-dimensional distributions.
      \end{theorem}
      \begin{proof}
        Because $EG(\eta_0(x))^2$ is finite, so is the spectral measure of the initial condition.
        Then we proceed exactly as in the proof of Theorem~\ref{heat-solution}.
      \end{proof}

      \begin{theorem}\label{heat-lw-solutions}
        Let $u(t) \= e^{t\Delta}G(\eta_0)$ be a solution to~\rfb{heat-lw}. Denote $R(|x|) = E\eta_0(0)\eta_0(x)$ and $m = \min\{n: c_n\neq0\}$,
        where $c_n$ are given by~\rfb{hermite}. Suppose the following
        \begin{enumerate}
          \item $R: [0,\infty)\to [0,1]$
          \item $R$ is non-increasing.
          \item $\lim_{x\to\infty}R(x)|x|^{1-k} = 1$ for some $k\in [1-\frac{1}{m},1)$.
        \end{enumerate}
        Then for the rescaled family of random fields
        \begin{equation*}
          u^T(t,\phi) = T^{\frac{m}{4}(1-k)} u(Tt)(\nu_{T^{-1/2}}\phi)
        \end{equation*}
        and the family
        \begin{equation*}
          w(t,\phi) = C_k^{m/2}\frac{c_m}{m!}\int_{\R^m} \frac{\FT{\phi}e^{-t(y_1+\ldots+y_m)^2}}{|y_1\cdots y_m|^{k/2}}\bigotimes_{i=1}^{m}W(dy_i)
        \end{equation*}
        we have $\wlim_{T\to\infty} u^T(t) \= w(t)$.
      \end{theorem}
      \begin{proof}
        Since we already established the scaling properties of the Laplace operator in Theorem~\ref{heat-scaling},
        we need only to study the behaviour of the rescaled orthogonal random measure.

        Denote $\gamma = \frac{m}{4}(1-k)$. Similarly to a calculation in the proof of Theorem~\ref{heat-scaling}, we have
        \begin{equation*}
          T^{\gamma}Z_n(T^{-1/2}A) =
          T^{\gamma-n/4}\int_{\R^n}\mathbbm{1}_{A}(y_1+\ldots+y_n)\bigotimes_{i=1}^{n}\sqrt{f(y_i/\sqrt{T})}\,W(dy_i).
        \end{equation*}
        Since
        \begin{equation*}
          R(|x|) = \int_\R e^{ix\xi} f(\xi)\,d\xi = 2\int_0^\infty \cos(x\xi)f(\xi)\,d\xi,
        \end{equation*}
        then because of a Tauberian theorem (see \cite[Thm.~4.10.3]{MR1015093}) we have
        \begin{equation*}
          \lim_{x\to 0} f(x)|x|^k = C_k,
        \end{equation*}
        for a constant $C_k$ (see~\cite{MR1015093} or~\cite{MR1632518} for details). Thus
        \begin{equation*}
          \lim_{x\to 0} T^{-k/2}f(x/\sqrt{T})|x|^k = C_k\quad\text{and}\quad
          \lim_{T\to\infty} T^{-k/2}f(x/\sqrt{T}) = C_k\frac{1}{|x|^k}
        \end{equation*}
        Hence for $n=m$ we obtain
        \begin{equation*}
          \lim_{T\to\infty}E\Big(T^{\gamma}Z_m(T^{-1/2}A) - C_k^{m/2}\int_{\R^m}\mathbbm{1}_{A}(y_1+\ldots+y_m)\bigotimes_{i=1}^{m}|y_i|^{-k/2}W(dy_i)\Big)^2 = 0.
        \end{equation*}
        We also have $f^{*n}(\xi) = 2\int_0^\infty\cos(\xi x)R(x)^n\,dx$ and therefore by using Lemma~\ref{convolutions} from Appendix
        we know that there exists $\delta$ such that for every $\xi\in[0,\delta)$ and every~$n$
        \begin{equation*}
          f^{*(n+1)}(\xi)< f^{*n}(\xi)
        \end{equation*}
        Therefore if $A\subset(-\delta,\delta)$ and for $n>m$
        \begin{equation*}
          E|Z_{n+1}(A)|^2 = \int_A f^{*(n+1)}(y)\,dy \leq \int_A f^{*n}(y)\,dy = E|Z_{n}(A)|^2
        \end{equation*}
        or, for every bounded $A$ and sufficiently large $T$,
        \begin{equation*}
          \frac{E|T^{\gamma}Z_{n+1}(T^{-1/2}A)|^2}{(n+1)!} \leq \frac{E|T^{\gamma}Z_{n}(T^{-1/2}A)|^2}{n!}.
        \end{equation*}
        This gives us the following estimate
        \begin{multline*}
          E\Big|\sum_{n>m}\frac{c_n}{n!} T^{\gamma}Z_n(T^{-1/2}A)\Big|^2 = \sum_{n>m}\frac{c_n^2}{(n!)^2} E|T^{\gamma}Z_n(T^{-1/2}A)|^2\\
          \leq \frac{E|T^{\gamma}Z_{m+1}(T^{-1/2}A)|^2}{(m+1)!} \sum_{n>m}\frac{c_n^2}{n!}.
        \end{multline*}
        The series $\sum_{n>m}\frac{1}{n!}c_n^2$ is convergent because of the Bessel inequality.
        Since
        \begin{equation*}
          \lim_{T\to\infty} \frac{E|T^{\gamma}Z_{m+1}(T^{-1/2}A)|^2}{(m+1)!} = 0
        \end{equation*}
        we obtain
        \begin{equation*}
          \lim_{T\to\infty} E\Big|\sum_{n>m}\frac{c_n}{n!} T^{\gamma}Z_n(T^{-1/2}A)\Big|^2 = 0.
        \end{equation*}
        Finally, because of Proposition~\ref{ccw-measures} we have
        \begin{multline*}
          \wlim_{T\to\infty} T^{1/4} u(Tt)(\nu_{T^{-1/2}}\phi)\\
          \= C_k^{m/2}\frac{c_m}{m!}\int_{\R^m} \frac{\FT{\phi}(y_1+\ldots+y_m)e^{-t(y_1+\ldots+y_m)}}{|y_1\cdots y_m|^{k/2}}\bigotimes_{i=1}^{m}W(dy_i).
          \qedhere
        \end{multline*}
      \end{proof}
    \end{subsection}
  \end{section}
  \begin{section}{Appendix}
    Here we present auxillary results required in the proof of Theorem~\ref{uniqueness} and a simple lemma we need to estimate an expression in the proof of
    Theorem~\ref{heat-lw-solutions}
    \begin{lemma}\label{torunczyk}
      Let $X$ be a locally convex separable metric linear space and $K$ be a compact set. If $Y$ is a dense linear subspace of $X$, then
      $C(K,Y)$ is dense in $C(K,X)$ (both with the uniform convergence topology).
    \end{lemma}
    \begin{proof}
      It follows from \cite[Thm. 4.2.14]{MR1851014} that $(X,Y)$ is an {\sf ANR}-pair in the sense of~\cite[p. 266]{MR1851014}. Then by \cite[Thm. 4.1.6]{MR1851014}
      there exists a homotopy $h$ on $[0,1]\times X$ such that $h(0)=\id_X$ and $h(t)(x)\in Y$ for every $t>0$ and $x\in X$ ($Y$ is homotopy dense in $X$).

      Let us take a function $g\in C(K,X)$ and put $u_t(k) = h(t)(g(k))$. Then we have $u_0 = g$ and for every $t>0$,  $u_t\in C(K,Y)$, and
      \begin{equation*}
        \lim_{t\to 0}\sup \|u_t - u_0\|_X =0.\qedhere
      \end{equation*}
    \end{proof}
    \begin{theorem}\label{separable-density}
      Assume {\rm(\ref{heat}a)} and {\rm(\ref{heat}b)}. The operator
      \begin{equation*}
        \B : \DT \to \CTL
      \end{equation*}
      has a dense image.
    \end{theorem}
    \begin{proof}
      Let $g \in \CTL$ and denote $T=\sup\supp g(s)$. Define $u$ by the formula
      \begin{equation*}
        u(t) = \int_0^t e^{(T-s)\LL^*} g(s)\,ds,\quad\text{for $t\leq T$}.
      \end{equation*}
      Then we have
      \begin{equation*}
        \B u(t) = g(t)\quad\text{and $u(t) \in C(\Tb,\F{\sigma_0})$,}
      \end{equation*}
      because $u$ is a mild solution to an appropriate (reversed in time) inhomogenous initial value problem (see \cite[Sec.~4.2]{MR710486}).
      We know that $C_c^\infty(\Tb\times\X)$ is dense in $C(\Tb,D)$.
      The space $\F{\sigma_0}$ is by definition a Hilbert completion of $D$ (see~\rfb{scalar-product}), and hence contains it as a dense linear subspace.
      Since $\F{\sigma_0}$ is separable, it follows from Lemma~\ref{torunczyk} that $C(\Tb,D)$ is dense in $C(\Tb,\F{\sigma_0})$.

      We may therefore take a sequence
      \begin{align*}
        &u_n\in C_c^\infty(\Tb\times\X)\subset \DT,\\
        &u_n\to u\quad \text{in } C(\Tb,\F{\sigma_0}).
      \end{align*}
      Then $\lim_{n\to\infty}\B u_n = \B u = g$ and thus
      \begin{equation*}
        \B : \DT \to \CTL
      \end{equation*}
      has a dense image.
    \end{proof}
    \begin{lemma}\label{convolutions}
      If $R:[0,\infty)\to[0,1]$ is a non-increasing and integrable function, then there exists $\delta>0$ such that for every $n$ and every $\xi\in[0,\delta)$
      we have
      \begin{equation*}
        \int_0^\infty \cos(\xi x)R^{n+1}(x)\,dx \leq \int_0^\infty \cos(\xi x)R^{n}(x)\,dx
      \end{equation*}
    \end{lemma}
    \begin{proof}
      Since the function $t\mapsto \int_0^t R(x)(1-R(x))\,dx$ is positive and non-decreasing there exists such $T$ that
      \begin{equation*}
        0\leq\frac{1}{2}\int_0^T R(x)(1-R(x))\,dx - \int_T^\infty R(x)(1-R(x))\,dx.
      \end{equation*}
      Hence for $\delta=\frac{\pi}{3T}$ and every $\xi\in[0,\delta)$
      \begin{equation*}
        0\leq\int_0^T \cos(\xi x)R(x)(1-R(x))\,dx - \int_T^\infty |\cos(\xi x)| R(x)(1-R(x)).
      \end{equation*}
      Now suppose that for some $n$ and and every $\xi \in [0,\delta)$ we have
      \begin{equation*}
        0\leq\int_0^T \cos(\xi x)R(x)^n(1-R(x))\,dx - \int_T^\infty |\cos(\xi x)| R(x)^n(1-R(x)).
      \end{equation*}
      Then since $R$ is non-increasing
      \begin{align*}
        &\int_0^T \cos(\xi x)R^{n+1}(1-R)\,dx \geq R(T)\int_0^T \cos(\xi x)R^n(1-R)\,dx\\
        &\geq R(T)\int_T^\infty |\cos(\xi x)| R^n (1-R)\,dx \geq \int_T^\infty |\cos(\xi x)| R^{n+1} (1-R)\,dx.
      \end{align*}
      Thus by the induction principle we prove that for every $n$ and every $\xi \in [0,\delta)$
      \begin{align*}
        0 &\leq \int_0^T \cos(\xi x)R(x)^{n}(1-R(x))\,dx - \int_T^\infty |\cos(\xi x)| R(x)^{n} (1-R(x))\\
        &\leq\int_0^\infty \cos(\xi x)R(x)^{n}(1-R(x))\,dx
      \end{align*}
      and finally
      \begin{equation*}
        \int_0^\infty \cos(\xi x)R(x)^{n+1}\,dx \leq \int_0^\infty \cos(\xi x)R(x)^{n}\,dx.
        \qedhere
      \end{equation*}
    \end{proof}
  \end{section}
  \section*{Acknowledgement}
    The author wishes to thank Grzegorz Karch for dedicated time and his invaluable help in preparing this paper.
  \bibliographystyle{amsalpha}
  \bibliography{evolution}
\end{document}